\documentclass[12pt]{article}
\usepackage[final]{epsfig}
\usepackage{graphics}
\usepackage{amsmath}
\usepackage{amsfonts}
\usepackage{latexsym}
\usepackage{amssymb}
\usepackage{graphicx}
\usepackage{url}
\usepackage{epstopdf}
\usepackage{hyperref}
\usepackage{xcolor}
\usepackage{comment}
\usepackage{tikz}
\usepackage{marginnote}
\usetikzlibrary{arrows.meta, positioning}

\newtheorem{lemma}{Lemma}[section]

\newtheorem{remark}[lemma]{Remark}

\newtheorem{theorem}{Theorem}

\newcommand{\eps}{{\varepsilon}}
\newcommand{\proofend}{$\Box$\bigskip}
\newcommand{\C}{{\mathbb C}}

\newcommand{\R}{{\mathbb R}}

\newcommand{\Z}{{\mathbb Z}}
\newcommand{\RP}{{\mathbb {RP}}}

\def\proof{\paragraph{Proof.}}
\def\Ker{\mathop{\rm Ker}}

\title{Symplectic billiards as Minkowski billiards}

\author{Peter Albers\footnote{
Institute for Mathematics,
Heidelberg University,
69120 Heidelberg,
Germany;
palbers@mathi.uni-heidelberg.de}
 \and 
Ana Chavez Caliz
\footnote{
Instituto de Matem\'aticas,
Universidad Nacional Aut\'onoma de M\'exico,
62210 Cuernavaca,
Mexico;
ana.chavez@im.unam.mx}
\and
 Serge Tabachnikov\footnote{
Department of Mathematics,
Pennsylvania State University,
University Park, PA 16802,
USA;
tabachni@math.psu.edu}
} 

\date{}

\begin{document}

\maketitle

\begin{abstract}
We establish a connection between Minkowski billiards and symplectic billiards, 
two classes of dynamical systems that have been studied largely independently. 
We show that the Minkowski billiard map can be described in symplectic terms via 
reduction from the canonical symplectic structure on $V \times V^*$, and that 
symplectic billiards can be viewed as a ``square root'' of a symplectic version 
of Minkowski billiards.

As an application, we recover several known results on symplectic billiards from 
the more general Minkowski setting, and extend some of them to higher dimensions 
and to periodic orbits of even period. In particular, we prove the existence of 
at least $(r-1)(n-1)$ $2r$-periodic symplectic billiard orbits in dimension $2n$.
\end{abstract}

\section{Introduction} \label{sect:intro}

The goal of this paper is to connect two ``stories" that seem to have evolved
independently of each other. The first ``story" concerns Finsler billiards,
introduced a while ago in \cite{GT}, more specifically their particular case,
Minkowski billiards. This subject remained in a dormant stage until its
connections with symplectic capacities and the Viterbo and Mahler conjectures 
were discovered and studied in \cite{ABKS,AO,AKO}. See also \cite{KR,Ru,VZ}. We
recall the definition and main properties of Minkowski billiards in Section
\ref{sect:Mink}.

The second ``story'' concerns symplectic billiards, introduced in \cite{AT}.
This is a billiard system inside a closed strictly convex smooth hypersurface
$M$ in a symplectic space defined as follows. Let $xy$ be a chord of $M$. The
symplectic billiard map $\varphi_S$ takes it to the chord $yz$ where $xz$ is parallel to 
the characteristic direction at point $y$, that is, the kernel of the restriction 
of the symplectic form to $T_yM$, see Figure \ref{reflection}.

\begin{figure}[ht]
\centering
\includegraphics[width=1.5in]{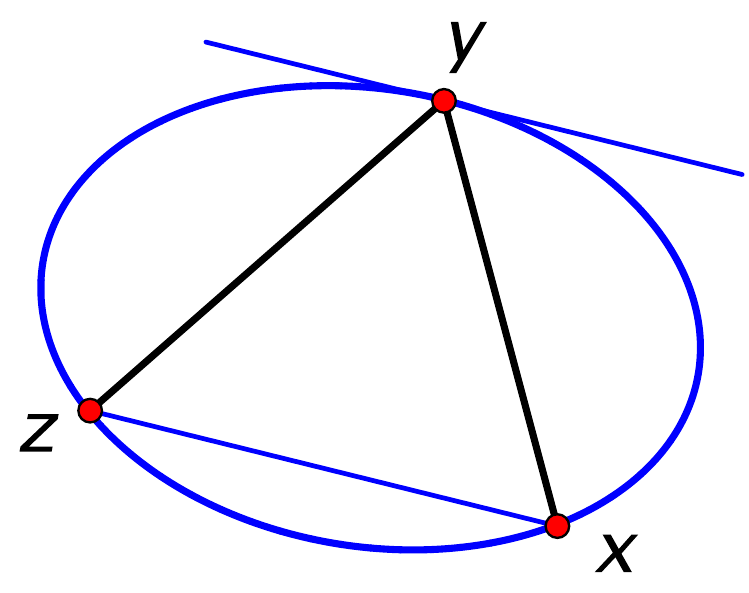}
\caption{The symplectic billiard reflection: $xy \mapsto yz$. }
\label{reflection}
\end{figure}

In addition to \cite{AT}, symplectic billiards were also studied in \cite{BB,BBN} 
independently of Minkowski billiards. A notable exception is the 
paper \cite{ALW} that implicitly connected these two subjects.
Here we make this connection explicit (Section \ref{subsect:sympM}) and show that 
some of the results on symplectic billiards follow from more general ones on Minkowski 
billiards; see Section \ref{subsect:ell}.   

\bigskip

{\bf Acknowledgments}. 

PA acknowledges funding by the Deutsche Forschungsgemeinschaft (DFG, German
Research Foundation) through Germany's Excellence Strategy EXC-2181/1 -
390900948 (the Heidelberg STRUCTURES Excellence Cluster), the Transregional
Collaborative Research Center CRC/TRR 191 (281071066). ACC was also funded by
the DFG through Project-ID 281071066 – TRR 191, and by Simons Foundation
International through the grant SFI-MPS-T-Institutes-00011977 JS. %LEC thanks the
%Heidelberg University and the Institut Henri Poincar\'{e} for their hospitality.
ST was supported by NSF grant DMS-2404535, Simons Foundation grant
MPS-TSM-00007747, and by a Mercator fellowship within the CRC/TRR 191. He thanks
the Heidelberg University and the Tel Aviv University for their hospitality.
The authors are grateful to Lael Costa for many helpful discussions.

\section{Minkowski billiards} \label{sect:Mink}

\subsection{Polar duality, Legendre transform, Minkowski metric}
\label{subsect:polar}
We start with a reminder of some basic notions of convex geometry. 

Let $V$ be a finite-dimensional vector space and $V^*$ its dual space. Let $Q\subset V$ be a
strictly convex closed $C^1$-smooth hypersurface that contains the origin in its
interior, i.e. the region it bounds in $V$. To every vector $q \in Q$ one assigns the covector $p \in V^*$
uniquely defined by the conditions $$ \Ker p = T_q Q\ \ \ {\rm and}\ \ \ q\cdot
p =1, $$ where the dot denotes the pairing of vectors and covectors. The set of these
covectors $p$ comprises a closed hypersurface $P \subset V^*$ that is also
strictly convex and contains the origin in its interior. The hypersurface $P$ is
called the {\it polar dual} to $Q$, written as $P=Q^*$, and the map ${\mathcal
L}: q \mapsto p$ is called the {\it Legendre transform}. The hypersurface polar
dual to $P$ is $Q$, and the Legendre transform $P \to Q$ is the inverse of the
Legendre transform $Q \to P$. Of course, here we use the canonical identification $V\cong V^{**}$.

A (not necessarily symmetric) Minkowski metric in a vector space $V$ is defined
by a strictly convex closed $C^1$-smooth hypersurface $I \subset V$ that
contains the origin in its interior as follows. The norm $|v|$ of a non-zero vector $v \in V$ is, by definition, 
$|v| := t$ where $t$ satisfies $\frac{v}{t} \in I$ and $t> 0$. The hypersurface $I$ is the unit
sphere in this metric; it is called the {\it indicatrix}, and its polar dual is
called the {\it figuratrix}. 

\subsection{Definition of Minkowski billiards, phase space, and generating function}
\label{subsect:def}
Recall the definitions and basic results concerning Minkowski billiards, see,
e.g., \cite{GT}.

Let $V$ be a vector space. The Minkowski billiard system is defined by two
closed $C^1$-smooth hypersurfaces that contain the origin in their interiors,
$Q\subset V$ and $P\subset V^*$. One thinks of $Q$ as a billiard table and of
$P$ as the figuratrix of a Minkowski metric, although the roles played by these
two hypersurfaces are totally symmetric.

The phase space ${\mathcal S}$ of the Minkowski billiard consists of pairs
$(q,v)$ where $q\in Q$ and $v\in I=P^*$ is a Minkowski unit vector with footpoint $q$
having an inward direction, see Figure \ref{refl1}. The Minkowski billiard map $\varphi_M$ is the composition of two
maps
$$
\varphi_M: (q,v) \mapsto (q_1,v) \mapsto (q_1,v_1), 
$$
where $q_1 \in Q$ is the intersection point of $Q$ with the line through $q$,
spanned by $v$, and $v_1$ is the Minkowski unit vector defined by the condition
$$ ({\mathcal L}(v_1)-{\mathcal L}(v))|_{T_{q_1} Q}  =0. $$

The Minkowski billiard map is illustrated in Figure \ref{refl1}, where, for ease
of visualization, the space $V$ is identified with $V^*$ by the Euclidean
structure, so  the Legendre transform of a point $p$ is represented by an
outward normal vector $N(p)$ to the respective hypersurface. 

\begin{figure}[ht]
\centering
\includegraphics[width=4in]{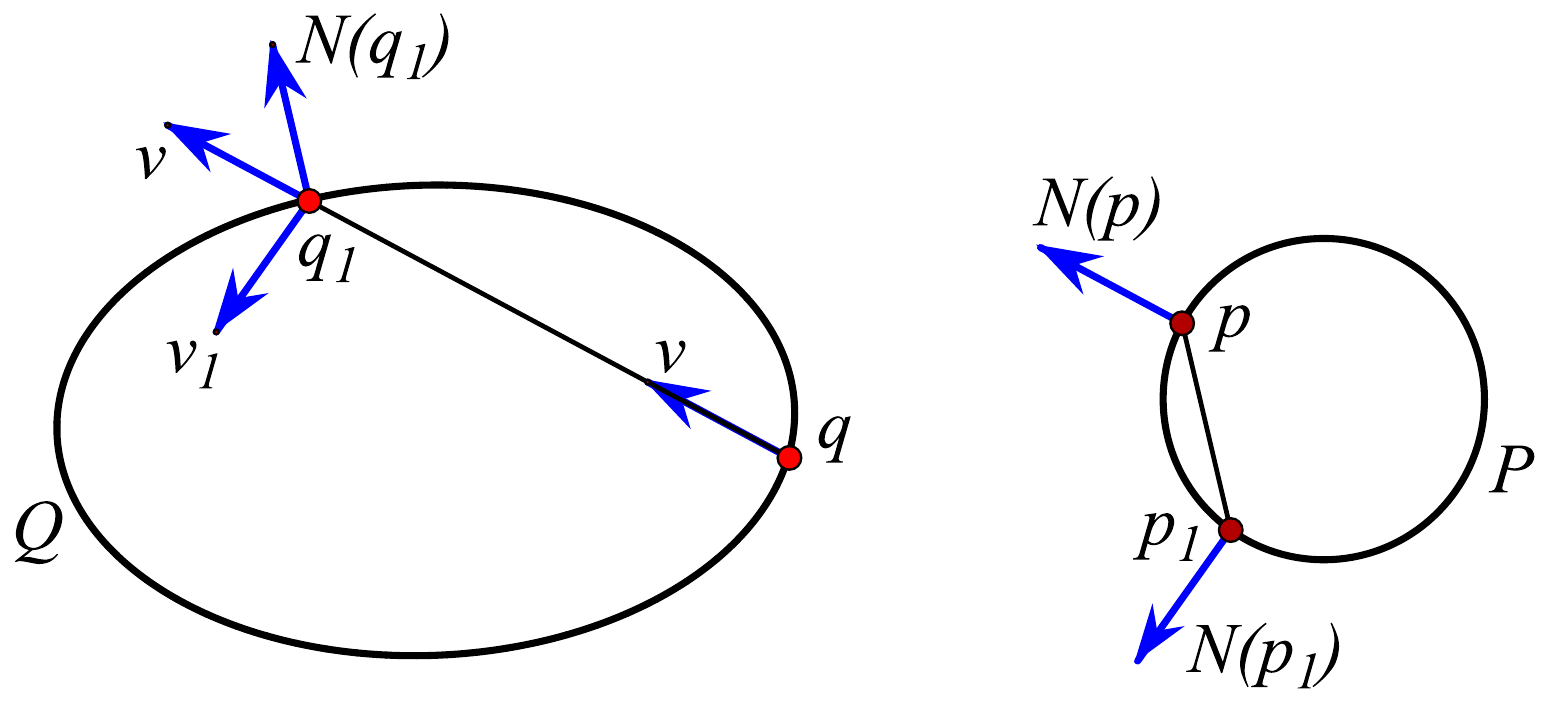}
\caption{Minkowski billiard reflection: $p={\mathcal L}(v), p_1={\mathcal
L}(v_1)$, and $(p_1-p)$ is antiparallel to ${\mathcal L}(q_1)$ (i.e., they are
parallel with opposite orientation). }
\label{refl1}
\end{figure}

If $P$ is the unit circle of a Euclidean metric, the reflection law is the
familiar law of optics: the angle of incidence equals the angle of reflection.
That is, the usual billiard is a particular case of Minkowski billiard.

\begin{figure}[ht]
\centering
\includegraphics[width=3.3in]{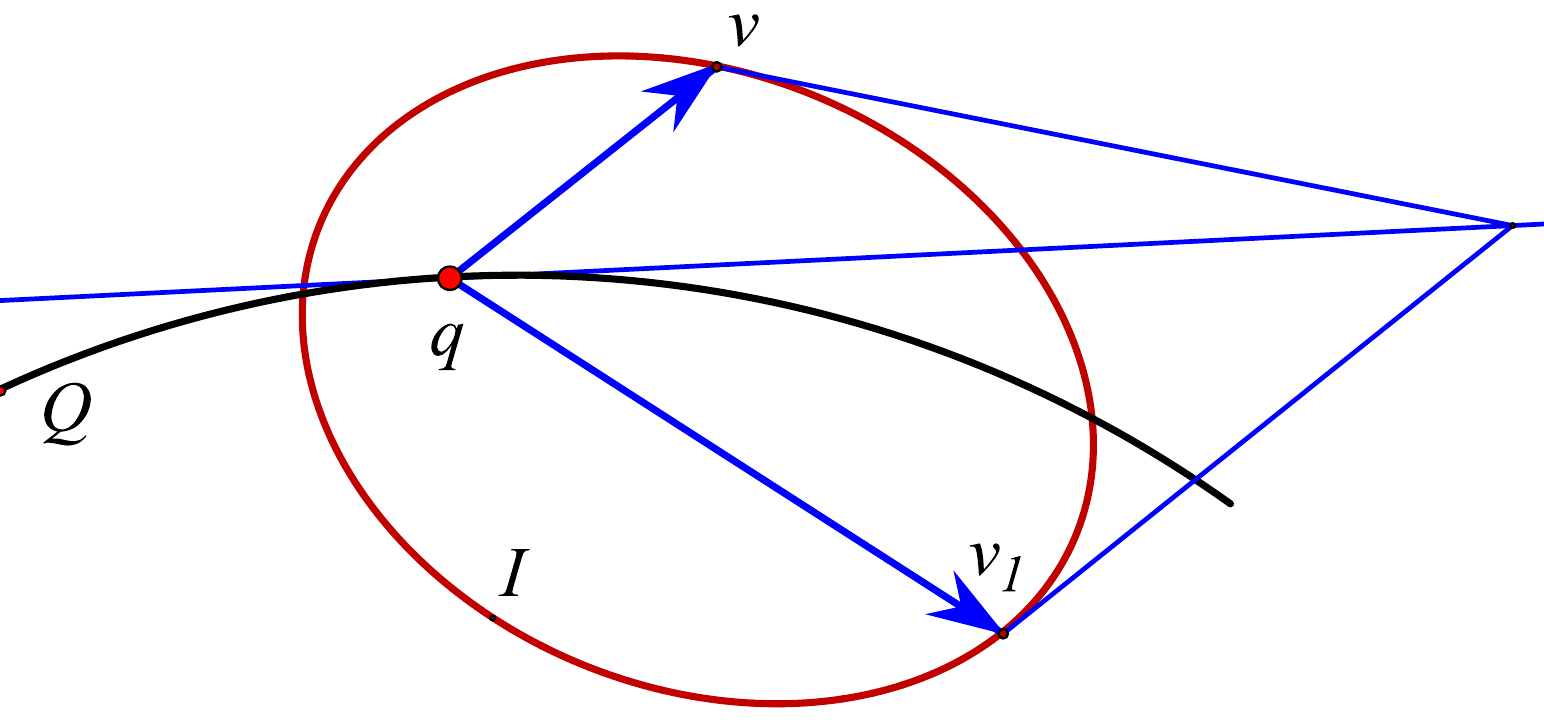}
\caption{Minkowski billiard reflection in dimension two: $I$ is the indicatrix
of the Minkowski metric, $v$ and $v_1$ are the incoming and outgoing unit
velocity vectors. The vector $v_1$ is determined by the triple intersection of the tangent line $T_qQ$ with the tangent lines to $I$ at $v$ and $v_1$.}
\label{refl2}
\end{figure}

Equivalently, this being the original motivation, the Minkowski billiard
reflection is defined by the same variational principle as in the usual
(Euclidean) billiards: a chord $qq_1$ reflects to $q_1q_2$ if the Minkowski
distance from $q$ to $q_1$ to $q_2$ is critical with respect to $q_1$:
\begin{equation} \label{eq:ref}
\frac{\partial}{\partial q_1} \left(|q-q_1| + |q_1-q_2|\right) =0.
\end{equation}
This translates to a  geometrically defined reflection law depicted in Figure
\ref{refl2}.

For example, consider the case when the indicatrix is a Kepler ellipse, that is,
an ellipse with the origin at a focus. Then, due to a theorem by de La
Hire\footnote{We thank A. Albouy for this  reference.}, the Minkowski reflection
law coincides with the optical law ``the angle of incidence equals the angle of
reflection,'' depicted in Figure \ref{ellipse}. See \cite{Ta} for relations with 
magnetic billiards and \cite{AK} for a multidimensional generalization of this fact.

\begin{figure}[ht]
\centering
\includegraphics[width=2.3in]{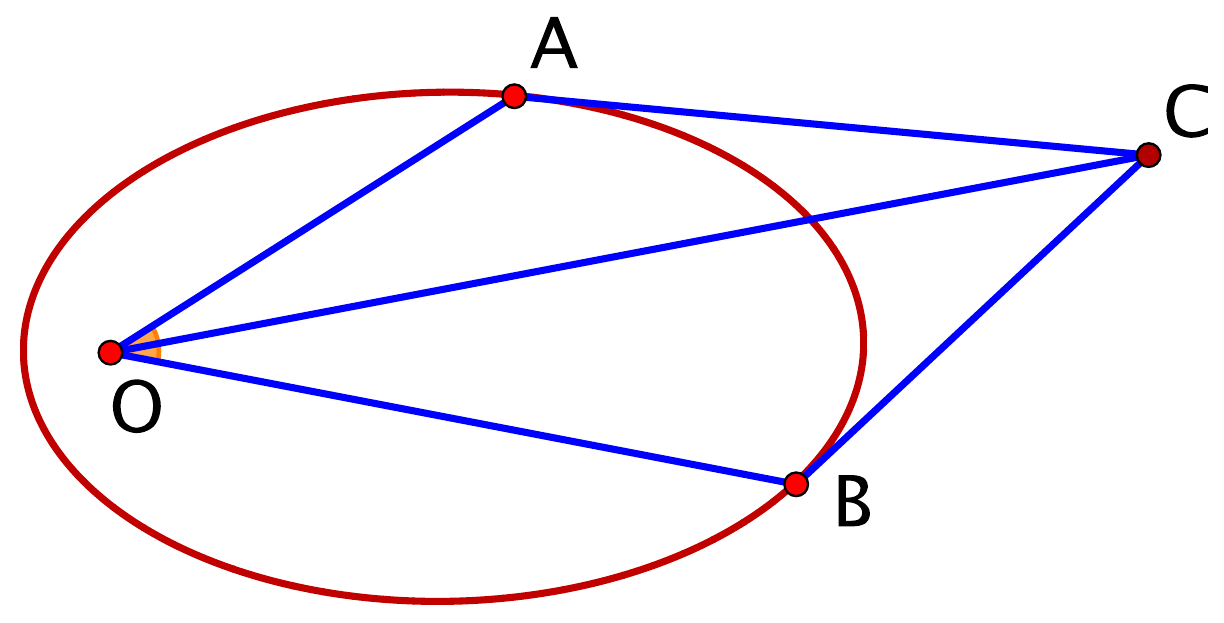}
\caption{For any pair of points $A$ and $B$ on an ellipse with focus $O$, one
has $\angle AOC = \angle BOC$.}
\label{ellipse}
\end{figure}

Let us return to the phase space ${\mathcal S}=\{(q,v)\}$. Let $v={\mathcal
L}(p)$ for $p \in P$. Since $v$ has an inward direction, ${\mathcal L} (q)
\cdot v < 0$. Thus we have the following, more symmetric, presentation of the
phase space
\begin{equation} \label{eq:phase}
{\mathcal S}= \{(q,p) \mid q\in Q, p \in P, {\mathcal L} (q) \cdot {\mathcal L}
(p) < 0\}.
\end{equation}
We write $a \sim b$ to denote  two vectors that are proportional with a positive
coefficient.

The next lemma shows that the space ${\mathcal S}$ defined in (\ref{eq:phase})
is invariant under the Minkowski billiard map.

\begin{lemma} \label{lm:cordef}
Let the Minkowski billiard map take $(q,v)$ to $(q_1,v_1)$, where $v={\mathcal
L}(p)$ and $v_1={\mathcal L}(p_1)$. Then ${\mathcal L} (q_1) \cdot {\mathcal L}
(p_1) < 0$.
\end{lemma}

\proof
Due to convexity, ${\mathcal L}(q_1) \cdot (q_1-q) > 0$, and since $q_1-q \sim
{\mathcal L}(p)$, one has ${\mathcal L}(q_1) \cdot {\mathcal L}(p) > 0$.
Likewise, ${\mathcal L}(p_1) \cdot (p_1-p) > 0$, and since $p_1-p \sim -
{\mathcal L}(q_1)$, one has  ${\mathcal L} (q_1) \cdot {\mathcal L} (p_1) < 0$,
as claimed.
\proofend

The closure $\bar {\mathcal S}$ is obtained by adding to ${\mathcal S}$ the set  
$$
{\mathcal O}=\{(q,p) | q\in Q, p \in P,
{\mathcal L} (q) \cdot {\mathcal L} (p) =0\}.
$$
By continuity, the Minkowski billiard map extends to this set as the identity
map. 

Notice that ${\mathcal O} \subset Q \times P$ is a hypersurface that partitions
$Q\times P$ into two diffeomorphic subspaces. Indeed, for every point $q\in Q$,
there is a unique point $q^* \in Q$ such that ${\mathcal L} (q)$ and  ${\mathcal
L} (q^*)$ are antiparallel. Then the involution $(q,p) \mapsto (q^*,p)$ is such
a diffeomorphism.  One can define a Minkowski billiard map on the other half
$$
\{(q,p) | q\in Q, p \in P, {\mathcal L} (q) \cdot {\mathcal L} (p) > 0\}
$$
as well, but it is conjugate to the inverse of the previously defined map.

To reiterate,  a configuration
\begin{equation} \label{eq:conf}
\ldots (q_{-1},p_{-1}), (q_0,p_0), (q_1, p_1),\ldots
\end{equation}
is an orbit of the Minkowski billiard map if \[q_{i+1} - q_i \sim {\mathcal
L}(p_i)\ \ {\rm and}\ \ p_{i+1} - p_i \sim -{\mathcal L}(q_{i+1}).\] It follows
that the configuration $\ldots (p_{-1}, -q_0), (p_0, -q_1), (p_1, -q_2), \ldots$
is an orbit of the Minkowski billiard map associated with the hypersurfaces $(P,
-Q)$ (Theorem 7.1 in \cite{GT} or Proposition 3.5 in \cite{KR}).

One also has the following lemma (Proposition 7.3 in \cite{GT}).

\begin{lemma} \label{lm:genf}
If the orbit (\ref{eq:conf}) is $n$-periodic, then it is a critical point of the
function
\begin{equation} \label{eq:gen}
\Phi:= \sum_{i=1}^n (q_{i+1}-q_i) \cdot p_i = - \sum_{i=1}^n q_i \cdot (p_i -
p_{i-1})
\end{equation} 
on the space ${\mathcal S}\times \dots \times {\mathcal S}$ ($n$ times), where
the indices are understood cyclically modulo $n$.
\end{lemma}

\proof
Let $v \in V$ be a fixed vector and let $p \in P$ be variable. Then $v \cdot p$
is extremal if and only if ${\mathcal L}(p) \sim \pm v$, where the plus sign
corresponds to the maximum and the minus sign to the minimum. Along an orbit,
one has $q_{i+1} - q_i \sim {\mathcal L}(p_i)$, hence  $\partial \Phi/\partial
p_i =0$. A similar argument, using the second formula for the function $\Phi$,
yields $\partial \Phi/\partial q_i =0$.
\proofend

\begin{remark} 
{\rm Although the definition of the polar dual hypersurface depends on the
    choice of the origin, the Minkowski billiard system is invariant under
    parallel translations of the hypersurfaces $Q$ and $P$. In particular, the
    function $\Phi$ from Lemma \ref{lm:genf} is invariant under such
    translations.
}
\end{remark}

\subsection{2-periodic orbits} \label{subsect:2per}

It is known that a strictly convex closed $C^1$-smooth hypersurface in
$n$-dimensional space has at least $n$ diameters, that is, the  chords
orthogonal to the hypersurface at both endpoints (see \cite{Ku}, Theorem 4).
These diameters are 2-periodic billiard trajectories. We extend this result to
Minkowski billiards.

\begin{theorem} \label{thm:diam}
The Minkowski billiard associated with hypersurfaces $Q \subset V$ and $P
\subset V^*$, where $\dim \: V =n$, has at least $n$ two-periodic orbits.
\end{theorem} 

\begin{figure}[ht]
\centering
\includegraphics[width=5.3in]{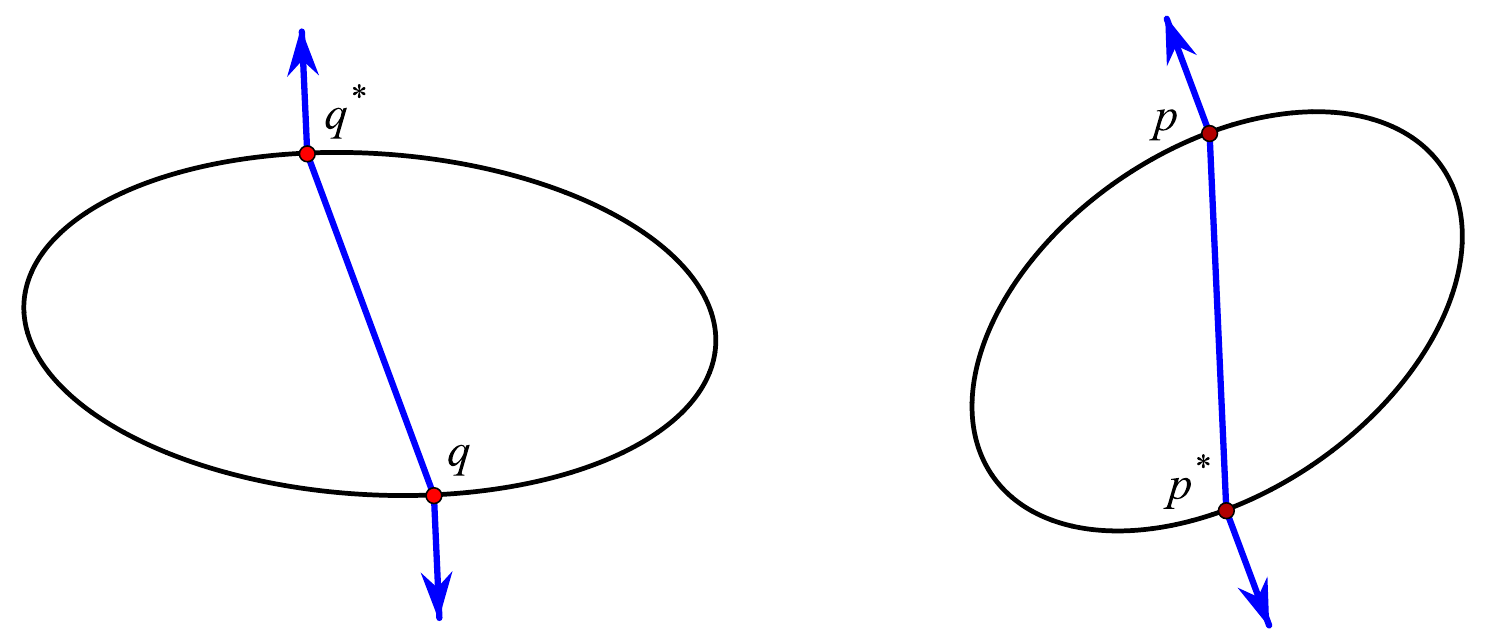}
\caption{A two-periodic orbit of the Minkowski billiard $(q, p, q^*, p^*)$.
The line $qq^*$ is parallel to the normals at the points $p$ and $p^*$, 
and the line $pp^*$ is parallel to the normals at the points $q$ and $q^*$.}
\label{2per}
\end{figure}

\proof
For convenience, identify $V$ with $\R^n$, so we use normals as Legendre
transforms. A 2-periodic orbit is presented in Figure \ref{2per}. 

For every point $q \in Q$, denote by $q^*$ the point where the normal $N(q^*)$
is antiparallel to that at $q$, and let $p^*$ have a similar meaning for points
$p \in P$. A 2-periodic orbit of the Minkowski billiard is a quadruple
$(q,p,q^*,p^*)$ such that $q^* - q \sim N(p)$ and $p^* - p \sim N(q)$.

Let $\Delta = \{(q,p,q_1,p_1)\ |\ q=q_1\ \ {\rm or}\ \ p=p_1\}$ be the
``diagonal.'' Consider the function
$$
\Phi(q,p,q_1,p_1) = (q_1-q)\cdot(p-p_1)
$$ 
from Lemma \ref{lm:genf} in the complement of $\Delta$. The group $\Z_2 \times
\Z_2$ acts on the quadruples $(q,p,q_1,p_1)$ by swapping $q$ with $q_1$ and $p$
with $p_1$. The function $\Phi$ is odd with respect to each factor $\Z_2$ and is
invariant under the diagonal action of $\Z_2$: $(q,p,q_1,p_1) \mapsto
(q_1,p_1,q,p)$.

A critical point of $\Phi$ off $\Delta$ is a quadruple $(q,p,q_1,p_1)$ such that
$q_1-q$ is collinear with $N(p)$ and $N(p_1)$, and $p_1-p$ is collinear with
$N(q)$ and $N(q_1)$. Hence $q_1=q^*, p_1=p^*$, and a 2-periodic orbit of the
Minkowski billiard corresponds to a $\Z_2 \times \Z_2$-orbit of this critical
point. These critical values of $\Phi$ are separated from zero.

Let $M \subset Q\times P\times Q\times P$ be the set of quadruples where $\Phi$
is positive. In particular, $q_1 \ne q$ and $p_1 \neq p$ in $M$. We claim that
$M$ is homotopically equivalent to $S^{n-1}$.

Let 
$$
N=\{(u,v)\ |\  u,v \in \R^n, |u|=|v|=1, u\cdot v > 0\}.
$$
Then $N$ is homotopically equivalent to $S^{n-1}$: for every unit vector $u$,
the set of unit vectors $v$ with $u \cdot v > 0$ is a hemisphere, a contractible
set. Next, we have a projection $M \to N$ that assigns to $(q,p,q_1,p_1)$ the
normalized to unit vectors $q_1-q$ and $p-p_1$. The fibers of this projection
are contractible as well, which proves the claim.

Let $\eps >0$ be small enough that there are no critical points of $\Phi$ in
$M$ with the critical values greater than or equal to $\eps$. Let $M_\eps$ be the
set where $\Phi \ge \eps$. Then $M_\eps$ is a manifold with boundary, and the
gradient of $\Phi$ has the inward direction along the boundary. Therefore we can
apply the Morse--Lusternik--Schnirelmann theory to  $M_\eps$. Since the diagonal
action of $\Z_2$ is free on $M_\eps$ and $\Phi$ is $\Z_2$-invariant, the number
of critical $\Z_2$-orbits is not less than the category of $M/\Z_2$, which is
homotopically equivalent to $\RP^{n-1}$. The Lusternik--Schnirelmann category of
this space is $n$, and the result follows.
\proofend

\begin{remark}
{\rm It is plausible that Theorem \ref{thm:diam} extends to pairs of wave fronts
    in the spirit of \cite{Pu1,Pu2}.
}
\end{remark}

\subsection{Symplectic structure} \label{subsect:symp}

The space $V \times V^*$ has a canonical symplectic structure $\Omega=dq \wedge
dp$, and this makes it possible to describe Minkowski billiard maps as
symplectic transformations.

In the recent literature on the relation between Minkowski billiards and the Viterbo
and Mahler conjectures, the dynamics of Minkowski billiards are described in terms 
of the Reeb flow on the boundary of the Lagrangian product of two convex bodies; 
see \cite{ABKS,AO,AKO,Ba,HO,KR,Ru1,Ru,Ru2}. We present below a somewhat different
treatment, similar to  the more familiar case of Birkhoff billiards.

Consider the following Melrose hexagonal diagram  useful in description of
billiard models (see \cite{Me} and, specifically, \cite{GT} concerning Finsler
billiards):
\begin{center}
\begin{tikzpicture}[
  node distance=2.6cm,
  every node/.style={font=\normalsize},
  arrow/.style={->, thick}
]
% Nodes
\node (VV)        at (0,3)   {$V \times V^{*}$};
\node (QVp)       at (-2.5,1.5) {$Q \times V^{*}$};
\node (VP)        at (2.5,1.5) {$V \times P$};
\node (QP)        at (0,0)   {$Q \times P$};
\node (TQ)        at (-2.5,-2) {$T^{*}Q$};
\node (TP)        at (2.5,-2) {$T^{*}P$};
% Arrows (upper hexagon)
\draw[arrow] (QVp) -- (VV);
\draw[arrow] (VP)  -- (VV);
\draw[arrow] (QP) -- (QVp);
\draw[arrow] (QP) -- (VP);
% Vertical arrows
\draw[arrow] (QVp) -- (TQ);
\draw[arrow] (VP)  -- (TP);
% Diagonal arrows from Q x P
\draw[arrow] (QP) -- (TQ);
\draw[arrow] (QP) -- (TP);
\end{tikzpicture}
\end{center}
Here the arrows having upward directions are inclusions. Let us describe the
vertical downward arrows as symplectic reduction maps. Consider the left one; the right one is treated the
same way.

The restriction of the symplectic form $\Omega$ on the hypersurface $q\times V^*
\subset V\times V^*$ has 1-dimensional kernels at all points, and the integral
curves of this line field are the characteristic curves.

\begin{lemma} \label{lm:char}
The characteristic line through point $(q,p)$ is $(q, p+ t{\mathcal L}(q))$
where $t \in \R$.
\end{lemma} 

\proof
Substitute the vector ${\mathcal L}(q) \partial/\partial p$ into $\Omega$ to
obtain the 1-form ${\mathcal L}(q) dq$. This 1-form vanishes on $Q$ by the
definition of the Legendre transform.
\proofend

The left vertical projection is the quotient of the space $Q \times V^*$ by the
characteristics, and by the lemma above, the quotient space is the cotangent
bundle $T^* Q$ with its canonical symplectic structure, obtained from $\Omega$
by symplectic reduction, that is, the restriction to a hypersurface $Q \times
V^*$ and factorizing by the kernel of this restriction.

Now consider the map $Q\times P$ to $T^* Q$, the composition of the inclusion
$Q\times P \subset Q \times V^*$ and the projection $Q \times V^* \to T^* Q$.
The image consists of points $(q,p)$ where $|p| = 1$ in the Minkowski norm
defined by $P$. Every point in the image has two preimages in
$Q\times P$, and an involution arises that interchanges these two preimages. By
construction, this involution preserves the restriction of $\Omega$ to $Q\times
P$, which is symplectic.

Likewise, one has another involution of $Q\times P$, similarly associated with
the map $Q\times P$ to $T^* P$. The composition of these two involutions is the
Minkowski billiard map, implying 

\begin{theorem} \label{thm:symp}
The Minkowski billiard map is symplectic: it preserves the restriction of the
canonical symplectic structure of $V\times V^*$ to $Q\times P$.
\end{theorem}

Let us mention that the same invariant symplectic structure can be obtained from
a generating function via the discrete Lagrangian systems formalism, see
\cite{Ve}. We briefly describe this approach.

Set $L(q,q_1)=|q_1-q|$. According to (\ref{eq:ref}), the reflection law $(q,q_1)
\mapsto (q_1,q_2)$ is given, in the vector notation, by
$$
L_2(q,q_1)+L_1(q_1,q_2) =0,
$$
where the subscripts denote the partial derivatives with respect to the first
and second variables. Take the exterior derivative
$$
L_{12}(q,q_1) dq+L_{22}(q,q_1) dq_1+L_{11}(q_1,q_2)dq_1+L_{12}(q_1,q_2)dq_2=0
$$
and wedge multiply by $dq_1$ on the left to obtain an invariant 2-form
$$
\omega:= L_{12}(q,q_1) dq_1\wedge dq = L_{12}(q_1,q_2) dq_2\wedge dq_1.
$$
Thus $\omega = d(L_1(q,q_1) dq)$.
 
On the other hand, tautologically, when a point $q$ moves toward a point $q_1$
with the Minkowski unit velocity $v$, the rate of decrease of the distance
$L(q,q_1)$ is unit. It follows that  $L_1(q,q_1) dq = -pdq$ (see Lemma 3.1 in
\cite{GT}), and hence $\omega=dq\wedge dp$, the same symplectic form as obtained
above via the Melrose diagram.

\section{Minkowski billiards in symplectic space and symplectic billiards}
\label{sect:Minksympl}

\subsection{Symplectic polar duality} \label{subsect:sympd}
Let $(V,\omega)$ be a symplectic vector space. Identify $V$ with its dual space
$V^*$ via the symplectic structure: to a vector $v\in V$ assign the covector
$\omega(\cdot, v)$.

Let $Q\subset V$ be a strictly convex closed $C^1$-smooth hypersurface that
contains the origin in its interior. Then the polar dual hypersurface comprises
the covectors $\omega(\cdot, R),\ R \in V$, such that
$$
\omega(u, R) = 0\ \ {\rm for\ all}\ \ u \in T_q Q,\ \ {\rm and}\ \ \omega(q, R)
= 1.
$$
That is, $R(q)$ is the Reeb field of the contact 1-form $\omega(q,\cdot)$ on
$Q$. 

By a slight abuse of language, we will continue to refer to the map $q \mapsto
R(q)$ as the (symplectic) Legendre transform. We will also denote by $Q^*$ the
{\it symplectic polar} of $Q$, which, under the identification  $V\simeq \C^n$,
is obtained  from the Euclidean polar by the composition with the complex
structure $J$.
%The map $q \mapsto R(q)$ is the Legendre transform, and its image, $Q^*$, is
%the {\it symplectic polar} of $Q$. If $V$ is identified with $\C^n$, then the
%symplectic polar is obtained from the respective Euclidean polar by the complex
%rotation $J$, given by multiplication by $\sqrt{-1}$.
See \cite{ACT,Be,BeB,BK,GG} for the recent study and applications of this
notion.

\subsection{Symplectic billiards as ``square root" of Minkowski
billiards}
\label{subsect:sympM}
We adapt the definition of the Minkowski billiard system to the symplectic
setting. In particular, we identify the space with its dual via the symplectic
structure. We will still call the resulting map Minkowski billiard map. 

Let $(V,\omega)$ be a symplectic vector space, and let $Q,P \subset V$ be two
strictly convex closed $C^1$-smooth hypersurfaces. The {\it Minkowski 
billiard map} with respect to $(Q, P)$ sends the pair $(q, p)$ to 
$(q_1, p_1)$ if 
\[q_{1} - q \sim R(p)\ \ {\rm and}\ \ p_1 - p \sim R(q_1),\]
as shown in Figure \ref{sympmap}.\footnote{Compared with the usual Minkowski 
billiard reflection (Figure \ref{refl1}), one observes a reversal of orientation. 
This sign difference reflects the fact that the symplectic polarity is not an 
involution: identifying vectors and covectors via the symplectic form satisfies 
$\omega^{-1}\circ\omega=-\mathrm{Id}$. See Lemma 2.5 in \cite{ACT}.}

The phase space is then
defined as
$$
{\mathcal S}= \{(q,p) | q\in Q, p \in P, \omega(R(q), R(p)) > 0\}, 
$$
compare with (\ref{eq:phase}), see \cite{ALW} for the planar polygonal case.

\begin{figure}[ht]
\centering
\includegraphics[width=5in]{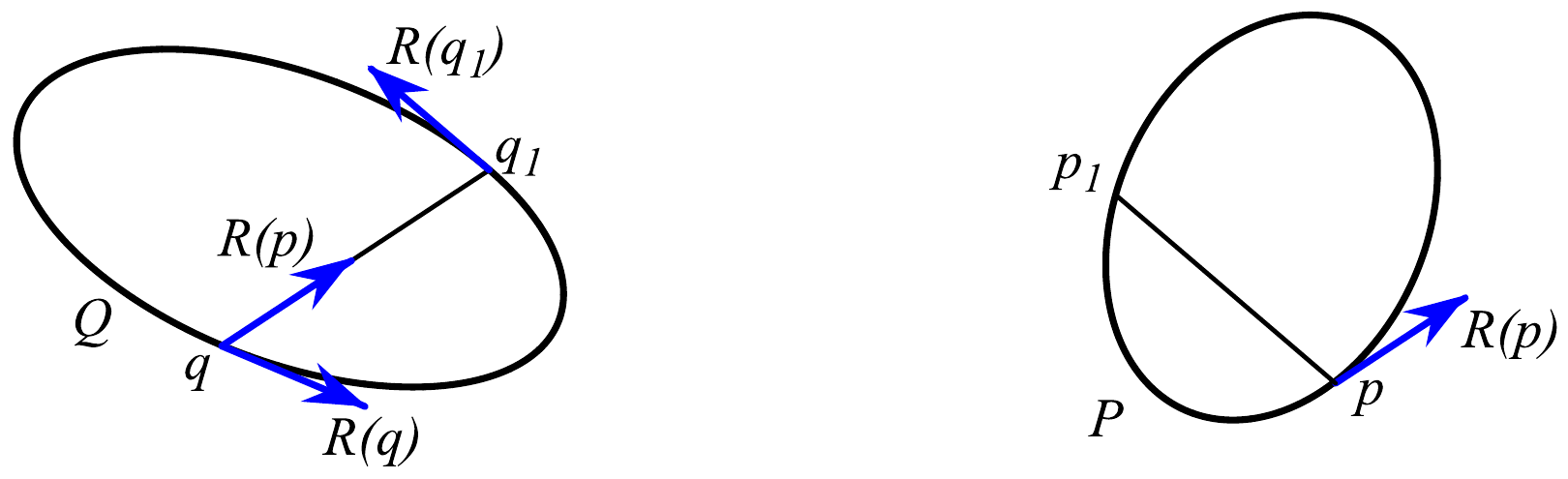}
\caption{Minkowski billiard map: $(q,p)\mapsto (q_1,p_1)$.}
\label{sympmap}
\end{figure}

Assume now that the hypersurfaces $Q$ and $P$ coincide.\footnote{In contrast,
the well-studied particular case of Minkowski billiards, related to the Viterbo
and Mahler conjectures, is when $P=Q^*$, see \cite{ABKS,AKO,HO}.} The respective
Minkowski billiard map is shown in Figure \ref{sympone}:
$(q,p)\mapsto (q_1,p_1)$. But this figure also presents the {\it symplectic
billiard map} $\varphi_S$: $(q,p)\mapsto(p,q_1)$, as defined in \cite{AT}. This leads to the 
following result.

\begin{figure}[ht]
\centering
\includegraphics[width=2.4in]{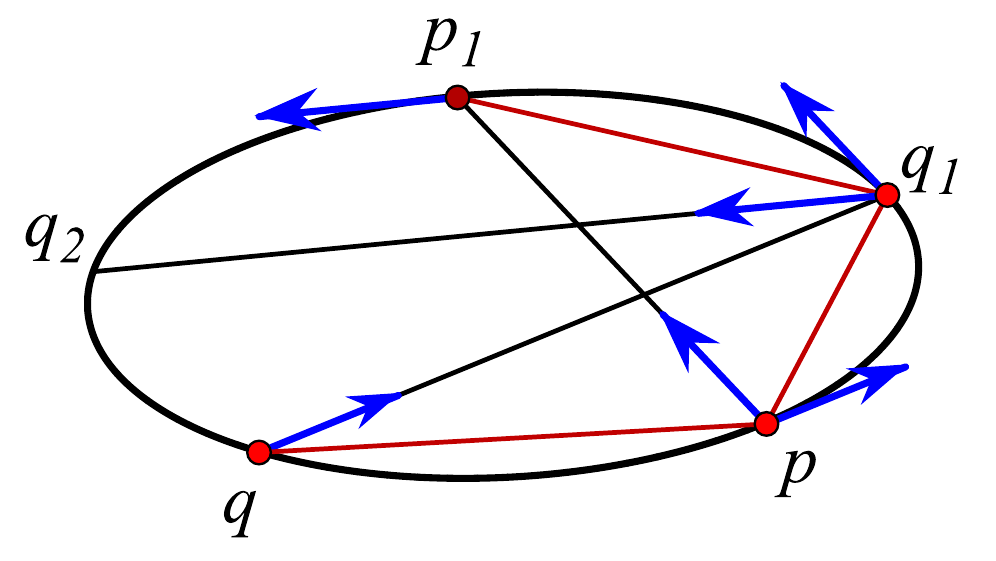}
\caption{Symplectic billiard map: $\varphi_S (q,p)= (p,q_1), \varphi_S (p,q_1)= (q_1,p_1)$; 
Minkowski billiard map: $\varphi_M (q,p)= (q_1,p_1)$.}
\label{sympone}
\end{figure}

\begin{theorem} \label{thm:root}
Assume that two strictly convex closed $C^1$-smooth hypersurfaces  $Q$ and $P$ in
symplectic space coincide. Then the symplectic billiard map with respect to $Q$ is 
the ``square root" of the Minkowski billiard map with respect to the 
pair $(Q, P)$, i.e., the latter is the second iteration of the former: $\varphi_M = \varphi_S^2$.
\end{theorem}

In particular, the generating function $\Phi$ in (\ref{eq:gen}) becomes
$$
\sum_{i=1}^n \omega (q_{i+1}-q_i, p_i) = -\sum_{i=1}^n
\big(\omega(q_i,p_i)+\omega(p_i,q_{i+1})\big)
$$
that is, the (negative) symplectic area of the polygon $(\ldots q_i , p_i,  q_{i+1}
\ldots)$.

\subsection{Some results on symplectic billiards revisited} \label{subsect:ell}
In this section we show how some results on symplectic billiards, obtained in
\cite{AT}, follow from our current setting. We also extend some of these results
to higher dimensions and greater periods.

\subsubsection{Four-periodic symplectic billiard orbits} 
It is proved in \cite{AT}, Theorem 12,  that the number of 4-periodic symplectic
billiard orbits in $2n$-dimensional symplectic space is not less than $2n$. In
view of the correspondence between symplectic billiards and Minkowski billiards,
this result follows from Theorem \ref{thm:diam} above.

\subsubsection{Ellipsoids}
Consider a Minkowski billiard system associated with a pair of hypersurfaces
$(Q,P)$, where $P$ is an ellipsoid centered at the origin. As we mentioned earlier, 
in this case one is dealing with the usual billiard in $Q$ with the Euclidean 
metric defined by $P$. If $Q$ is also an ellipsoid, this billiard system is 
completely integrable, see, e.g., \cite{Ve}.

Now consider the symplectic billiard in an ellipsoid. Then we are in the setting
described in the preceding paragraph, and Theorem \ref{thm:root} implies that
if $(\ldots q,p,q_1,p_1,q_2,p_2,\ldots)$ is a trajectory of the symplectic
billiard, then $(\ldots q, q_1,q_2,\ldots)$ is a trajectory of the respective
Minkowski billiard. In particular, the symplectic billiard map in an ellipsoid is
completely integrable. This is another result of \cite{AT}, see Theorem 10
therein. 

\subsubsection{Radon curves and Minkowski bodies of constant width}

The Birkhoff billiard in a curve of constant width possesses an invariant curve
consisting of 2-periodic points, the diameters of the curve, i.e., the chords
that are orthogonal to the curve at both endpoints. Likewise for bodies of
constant width in higher dimensions. 

Consider the Minkowski billiard defined by hypersurfaces $Q\subset V$ and
$P\subset V^*$, where dim $V=n$. Let 
$$
\bar Q := Q \oplus (-Q),\ \bar P := P \oplus (-P)
$$
be the symmetrizations of $Q$ and $P$, respectively (where $\oplus$ denotes the
Minkowski sum).

\begin{lemma} \label{lm:width}
If $\bar P=\bar Q^*$ then the Minkowski billiard has an $(n-1)$-parameter family
of 2-periodic orbits: such an orbit originates at every point $q\in Q$ and
exists for every direction.
\end{lemma} 

\proof
Referring to Figure \ref{2per} and the notations of the proof of Theorem
\ref{thm:diam}, one needs to do the following: given $q \in Q$, define $p$ by
$N(p) \sim q^* - q$, and check that $p^* - p \sim N(q)$.

First, observe that if the above property holds for $\bar Q$ and $\bar P$ then
it holds for $Q$ and $P$ as well. Indeed, $q^*-q$ is the point of $\bar Q$ with
the same (co)normal as at point $q^*$ of $Q$, and likewise for $P$. Thus we may
assume that $Q$ and $P$ are centrally symmetric to start with.

With this assumption, $N(p) \sim q^* - q$ means that $p$ is the Legendre
transform of $q^*=-q$ and, similarly, $p^* - p \sim N(q)$ means that $Q$ is the
Legendre transform of $p^*=-p$. These two conditions are equivalent because the
Legendre transform is an involution.
\proofend

For example, if $P$ is a circle, then we are dealing with Birkhoff billiards. If
$Q$ is a curve of constant width, then $\bar Q$ is also a circle, and $\bar
Q^*=\bar P$. More generally, one can take, as $Q$ and $P$, two curves of
constant width.

It was observed in \cite[Section 2.4.2]{AT} that the symplectic billiard map with 
respect to a Radon curve possesses an invariant curve consisting of 4-periodic 
orbits. Radon curves are self-dual, that is, they satisfy the condition of Lemma
\ref{lm:width}, therefore this lemma implies the said observation concerning
Radon curves in \cite{AT}.

Note also that a recent paper \cite{BeB} contains a related result: for a
symplectically self-polar convex body, the outer billiard has an invariant
hypersurface consisting of centrally symmetric 4-periodic orbits. The relation
to symplectic billiard is that the midpoints of an outer billiard 4-periodic
orbit form a 4-periodic orbit of the symplectic billiard inside the same body.

\subsubsection{Periodic orbits} 

Concerning periodic orbits of symplectic billiards in $2n$-dimensional space,
the following two results are proved in \cite{AT}:

\noindent
{\it
1) For every $r \ge 2$, the symplectic billiard map has an $r$-periodic
trajectory;\\
2) The number of 3- and 4-periodic symplectic billiard trajectories is not less
than $2n$.
}\\
\noindent
(Here one is counting the orbits of the dihedral group $D_r$ acting on $r$-periodic
trajectories).

However, Minkowski billiards are a particular case of Finsler billiards.
Periodic orbits of the latter were studied in \cite{BHTZ}, and one can apply the
results obtained therein to symplectic billiards. 

The main result of \cite{BHTZ}, Theorem 1.2, specialized to the Minkowski
setting, is as follows. Consider the  Minkowski billiard map in a strictly convex
closed $C^1$-smooth hypersurface in $d\ge 3$-dimensional vector space. Let $r\ge
3$ be a prime number. Denote by $N_M(d,r)$ the number of $r$-periodic Minkowski billiard
orbits, where two orbits that differ by a cyclic permutation of points are
considered to be the same (i.e., here one counts the orbits modulo the action of
the cyclic group $\Z_r$). Then $N_M(d,r) \ge (r-1)(d-2)+1$.

We note that this result holds for every strictly convex closed $C^1$-smooth hypersurface,
that is, it is obtained via Lusternik–Schnirelmann theory; if the hypersurface is generic
in the sense that the Morse theory applies, and $d$ is even, the result is stronger: $N_M(d,r) \ge (r-1)d$.
We also note that these lower bounds apply to non-reversible Finsler billiards, such as the magnetic ones);
one expects to have stronger results for reversible ones. 

Consider a strictly convex closed $C^1$-smooth hypersurface in symplectic space $\R^{2n}$ and 
denote by $N_S(2n,r)$ the number of $r$-periodic symplectic billiard orbits, where, as before,  two orbits that differ by a cyclic permutation of points are considered to be the same. 

\begin{theorem} \label{thm:bd}
Let $r\ge 3$ be a prime number. Then $N_S(2n,2r) > (r-1)(n-1)$.
\end{theorem}

Although this estimate is likely not sharp, it is linear in the
period and dimensions, similarly to the case of multi-dimensional Birkhoff billiards, see \cite{FT}.

\proof
The following correspondence 
\begin{equation} \label{eq:cor}
((q_1, p_1) (p_1,q_2) (q_2, p_2) \cdots (q_r, p_r)) \mapsto ((q_1,p_1) (q_2,p_2) \cdots (q_r,p_r))
\end{equation}
is a bijection between $2r$-periodic symplectic billiard orbits (on the left) and $r$-periodic Minkowski billiard orbits (on the right). The group $\Z_r$ naturally acts on both sides of (\ref{eq:cor}), 
\begin{equation}\nonumber
\begin{aligned}
((q_1, p_1) (p_1,q_2) (q_2, p_2) \cdots (q_r, p_r)) &\mapsto ((q_2, p_2) \cdots (q_r, p_r) (q_1, p_1)(p_1,q_2))\\
((q_1,p_1) (q_2,p_2) \cdots (q_r,p_r)) &\mapsto ((q_2,p_2) \cdots (q_r,p_r) (q_1,p_1) )
\end{aligned}
\end{equation}
and the map is clearly $\Z_r$ equivariant. 

The generator of the cyclic group $\Z_{2r}$ acts on $2r$-periodic symplectic billiard orbits as follows:
$$
((q_1, p_1) (p_1,q_2) (q_2, p_2) \cdots (q_r, p_r)) \mapsto ((p_1,q_2)(q_2, p_2) \cdots (q_r, p_r) (q_1, p_1)).\\
$$
This $\Z_{2r}$ action, of course, yields the same symplectic billiard orbit but corresponds, via the above map, to a different $r$-periodic Minkowski billiard orbit, namely, $(p_1,q_2), (p_2,q_3),\cdots, (p_r,q_1)$, see Figure \ref{sixper}.

%\marginpar{added this paragraph}
Let us rephrase this slightly using Theorem \ref{thm:root}, i.e., the fact that the symplectic billiard map is  the square root of the Minkowski billiard map: $\varphi_M = \varphi_S^2$. This implies that the $2r$-periodic symplectic billiard orbit starting at $(q_1,p_1)$ gives rise to two $r$ periodic orbits of the Minkowski billiard map. One starting at $(q_1,p_1)$ and another starting at $\varphi_S(q_1,p_1)=(p_1,q_2)$. 

It follows that 
$$
N_S(2n,2r) \geq \frac12 N_M(2n,r) \geq \frac{(r-1)(2n-2) +1}{2} > (r-1)(n-1),
$$
as claimed.
\proofend

\begin{figure}[ht]
\centering
\includegraphics[width=2.4in]{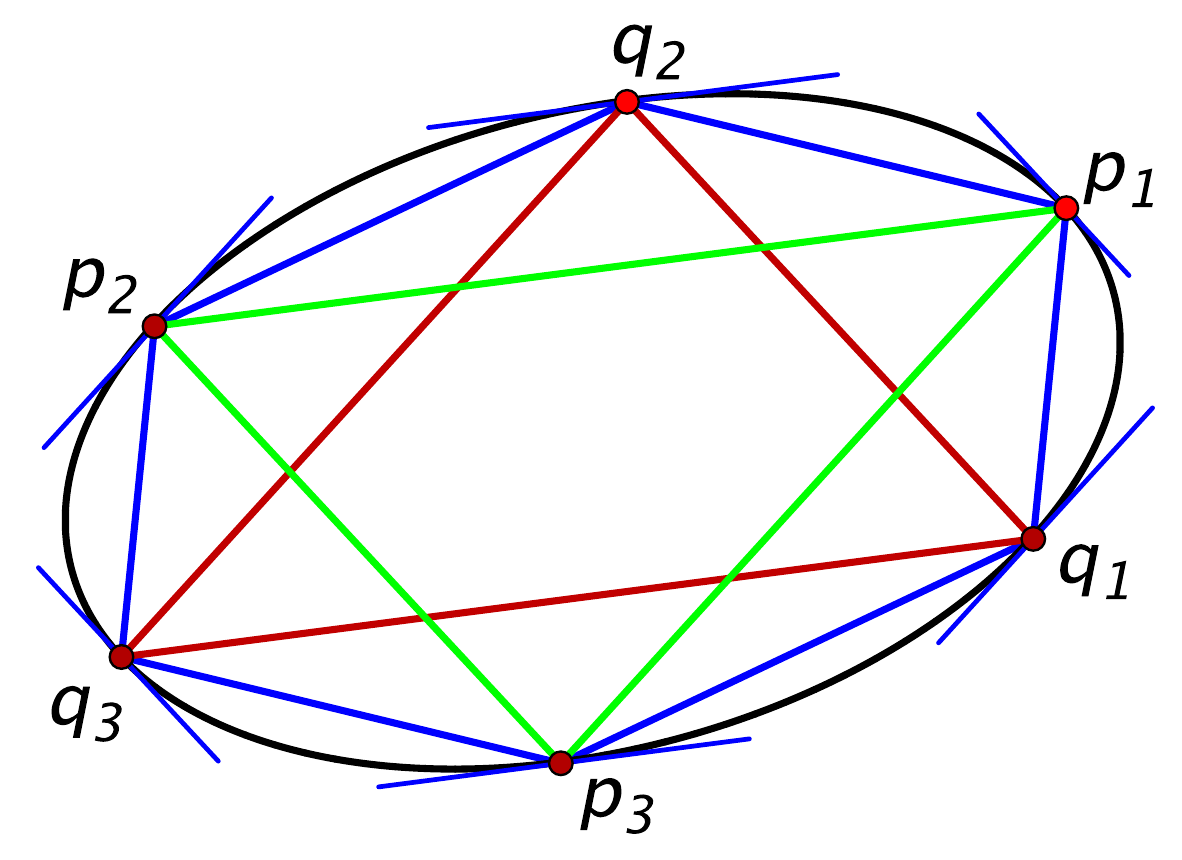} 
\caption{The correspondence between symplectic and Minkowski billiard orbits: 
the 6-periodic symplectic billiard orbit 
$((q_1, p_1) (p_1,q_2) (q_2, p_2) (p_2, q_3) (q_3, p_3) (p_3, q_1))$
corresponds to a 3-periodic Minkowski billiard orbit
$((q_1,p_1) (q_2,p_2) (q_3,p_3))$. The 6-periodic symplectic billiard orbit 
$
((p_1,q_2) (q_2, p_2) (p_2, q_3) (q_3, p_3) (p_3, q_1) (q_1, p_1)),
$
 obtained by a cyclic permutation of the points, corresponds to another 3-periodic Minkowski billiard orbit, 
$((p_1,q_2) (p_2, q_3) (p_3, q_1))$.
}
\label{sixper}
\end{figure}%\marginpar{there should be p,q in the Minkowski orbit}

\end{document}